\documentclass[11pt]{amsart}
\usepackage{geometry}                % See geometry.pdf to learn the layout options. There are lots.
\usepackage[parfill]{parskip}    % Activate to begin paragraphs with an empty line rather than an indent
\usepackage{graphicx}
\usepackage{amssymb}
\usepackage{epstopdf}
\usepackage{stmaryrd}
\usepackage{subfig}
\usepackage{tikz}
\usepackage{pgflibraryarrows}
\newcommand{\TT}[1]{\mathbb T_{#1}}
\newcommand{\rec}{{\talloblong}}
\newtheorem{lemma}{Lemma}
\newtheorem{theorem}[lemma]{Theorem}
\newtheorem{proposition}[lemma]{Proposition}
\newtheorem{conjecture}[lemma]{Conjecture}
\newtheorem{conjecturethree}{Conjecture}
\newcommand{\plussign}{$ \begin{minipage}{1.7em}\begin{tikzpicture}[thick]
\draw (0,.3)--(.6,.3);
\draw (.3,0)--(.3,.6);
\end{tikzpicture}\end{minipage}$ }
\newcommand{\leftplus}{$ \begin{minipage}{1.7em}\begin{tikzpicture}[thick]
\draw (0,.4)--(.3,.4);
\draw (.3,.2)--(.6,.2);
\draw (.3,0)--(.3,.6);
\end{tikzpicture}\end{minipage}$ }
\newcommand{\rightplus}{$ \begin{minipage}{1.7em}\begin{tikzpicture}[thick]
\draw (.3,.4)--(.6,.4);
\draw (0,.2)--(.3,.2);
\draw (.3,0)--(.3,.6);
\end{tikzpicture}\end{minipage}$ }

\newcommand{\Out}{\operatorname{Out}}
\newcommand{\im}{\operatorname{im}}
\newcommand{\z}{\mathbb Z_2}

\title{On the number of tilings of a square by rectangles}
\author{Jim Conant}
\author{Tim Michaels}
\begin{document}
\begin{abstract}
We develop a recursive formula for counting the number of  rectangulations of a square, i.e the number of combinatorially distinct tilings  of a square by rectangles.  Our formula specializes to give a formula counting generic rectangulations, as analyzed by Reading in \cite{reading}. Our computations agree with \cite{reading} as far as was calculated and extend to the non-generic case. An interesting feature of the number of rectangulations is that it  appears to have an $8$-fold periodicity modulo $2$. We verify this periodicity for small values of $n$, but the general result remains elusive, perhaps hinting at some unseen structure on the space of rectangulations, analogous to Reading's discovery that generic rectangulations are in 1-1 correspondence with a certain class of permutations.
Finally, we use discrete Morse theory to show that the space of tilings by $\leq n$ rectangles is homotopy-equivalent to a wedge of some number of $(n-1)$-dimensional spheres. Combining this result with formulae for the number of tilings, the exact homotopy type is computed for $n\leq 28$. 
\end{abstract}
\maketitle

\section{Introduction}
Let $\TT{n}$ be the topological space of tilings of a fixed unit square by $\leq n$ rectangles, topologized via the Hausdorff metric applied to the $1$-skeleta of tilings, including the boundary. Our purpose in this paper is to analyze the topology, or more precisely, the homotopy type, of this space. We were motivated to study the topology of $\TT{n}$ after discussions with Ted Stanford, who had been looking at the space of tilings of a torus by rectangles \cite{stanford}.  In Proposition~\ref{prop:cell}, we show that the space $\TT{n}$ has a natural cell structure with cells in correspondence with combinatorial types of tilings (called rectangulations), and is thus akin to many spaces arising in topology with cells parameterized by combinatorial objects.  For example the Grassmannian, the tropical Grassmannian $\mathcal G_{2,n}$, Teichm\"{u}ller space, and Outer space  have cell decompositions by Schubert cells, phylogenetic trees, marked ribbon graphs and marked graphs respectively \cite{grass,tropical,teich,outer}. The homotopy type of the latter two, modulo the actions of the mapping class group and $\Out(F_n)$ respectively, remains an unsolved problem in general.

Our first main result 
 uses discrete Morse theory to establish that $\TT{n}$ is a wedge of top-dimensional spheres.

\begin{theorem}\label{thm:wedge}
There is a homotopy equivalence $\TT{n}\simeq \bigvee_{i=1}^{k_n} S^{n-1}$, for some nonnegative integer $k_n$.
\end{theorem}

Thus knowledge of the numbers $k_n$ is necessary to complete the calculation of the homotopy type of $\TT{n}$.  These are listed for $n\leq 28$ in Proposition~\ref{prop:kn} below.

Two tilings of a unit square by rectangles are said to be \emph{combinatorially equivalent} if there is a homeomorphism taking one to the other that fixes the four corners and which preserves the horizontality and verticality of line segments. Following \cite{reading}, we let a \emph{rectangulation} be a combinatorial equivalence class of tilings. In an earlier version of this paper, we dropped the requirement  that the homeomorphism preserves horizontality and verticality, which leads to a different definition (not the one we had in mind) as the following example shows.
$$
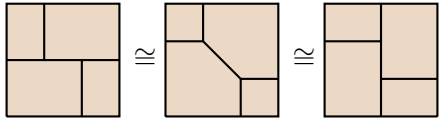
\begin{minipage}{1.6cm}
\begin{tikzpicture}[thick,scale=.5]
\pgfsetfillopacity{0.3}
\draw[fill=brown](0,0)--(0,3)--(3,3)--(3,0)--(0,0);
\draw(2,0)--(2,1.5)--(3,1.5);
\draw(0,1.5)--(1,1.5)--(1,3);
\draw(1,1.5)--(2,1.5);
\end{tikzpicture}
\end{minipage}
\cong
\begin{minipage}{1.6cm}
\begin{tikzpicture}[thick,scale=.5]
\pgfsetfillopacity{0.3}
\draw[fill=brown](0,0)--(0,3)--(3,3)--(3,0)--(0,0);
\draw(2,0)--(2,1)--(3,1);
\draw(0,2)--(1,2)--(1,3);
\draw(1,2)--(2,1);
\end{tikzpicture}
\end{minipage}
\cong
\begin{minipage}{1.6cm}
\begin{tikzpicture}[thick,scale=.5,rotate=90]
\pgfsetfillopacity{0.3}
\draw[fill=brown](0,0)--(0,3)--(3,3)--(3,0)--(0,0);
\draw(1,0)--(1,1.5)--(0,1.5);
\draw(2,3)--(2,1.5)--(3,1.5);
\draw(1,1.5)--(2,1.5);
\end{tikzpicture}
\end{minipage}
$$

As we mentioned above, the space $\TT{n}$ is a cell complex, where the cells are in 1-1 correspondence with rectangulations. A vertex of a rectangulation where four rectangles meet in a corner is called \emph{singular}. The dimension of a cell corresponding to a rectangulation with $m$ tiles and $s$ singular vertices is $m-s-1$; see Proposition~\ref{prop:cell}. So, using the Euler characteristic,  the numbers $k_n$ are determined by the numbers of cells in each dimension. This amounts to being able to count the number $t_{m,s}$ of rectangulations with $m$ tiles and $s$ singular vertices. A closed formula for this would be nice and is still an open problem. In \cite{reading}, the numbers $t_{m,0}$ are shown to count a certain kind of permutation of $m$ letters, called \emph{$2$-clumped permutations.} In particular, the striking fact that $t_{1,0}=1,t_{2,0}=2,t_{3,0}=6,$ and $t_{4,0}=24$ arises since the first  permutations which are not $2$-clumped occur when $m=5$. In any event, Reading's beautiful result does not easily give a formula for $t_{m,0}$ and also does not apply to $t_{m,s}$ for $s\geq 1$. In this paper, we give a recursive formula to calculate $t_{m,s}$ which is easily implemented on a computer, and which we used to calculate these numbers for $m\leq 28$.

More precisely, we need to refine $t_{m,s}$ by also specifying how many edges hit the interior of the right side of the unit square. We say that a rectangulation is $r$-heavy if there are $r$ such edges.
 Let $t_{m,r,s}$ be the number of $r$-heavy rectangulations with $m$ tiles and $s$ singular vertices.

\begin{theorem}\label{thm:recursion}
The number of $r$-heavy rectangulations with $m$ tiles and $s$ singular vertices is given by the formula
$$
t_{m,r,s}=\sum_{\bar m,\bar r, \bar s, c}
(-1)^{c+1}\binom{\ell-1}{c-1} 
\binom{{\bar r}+2-\ell}{c}
\binom{\ell-c}{{\Delta s}}\binom{{\Delta m}-c-{\Delta s}+\ell-1}{\ell-1}
t_{{\bar m},{\bar r},{\bar s}}
$$
where the sum ranges over the set of $(\bar m,\bar r,\bar s,c)$ such that  $1\leq {\bar m}\leq{m-1}$, $0\leq {\bar r}\leq{\bar m-1}$, $0\leq{\bar s}\leq {s}$ and $1\leq {c}\le{\lceil({\bar r}+1)/2\rceil}$.
Here ${\Delta s}=s-{\bar s}$, $\Delta m=m-{\bar m}$, $\Delta r=r-\bar r$, and $\ell=\Delta m-\Delta r$.

The base of the recursion is given by $t_{k,k-1,0}=1$ for $k\geq 1$.
\end{theorem}

The above recursion yields the  data in Figure~\ref{fig:data}, where $\displaystyle t_{m,s}=\sum_{r=0}^{m-1}t_{m,r,s}$, and $\displaystyle t_m=\sum_{s\geq 0} t_{m,s}$ is  the total number of rectangulations with $m$ tiles.  %Although $t_m$ does not have topological significance, since it mixes the count of cells of different dimensions, it is an interesting combinatorial quantity.
\begin{figure}
\begin{center}
\begin{tabular}{c|cccccccccccc}
$m$&1&2&3&4&5&6&7&8&9&10&11&12\\
\hline
$t_{m,0}$&1&2&6&24&116&642&3938&26194&186042&1395008&10948768&89346128\\
$t_{m,1}$&0&0&0&1  & 12  &114&1028&9220 &83540   &768916  & 7200852  &68611560\\
$t_{m,2}$&0&0&0&0  & 0    &2     &48     &770   &10502   &132210  &1593934   &18755516\\
$t_{m,3}$&0&0&0&0  & 0    &0     &0       &10      &348       &7680       &137940     &2206972\\
$t_{m,4}$&0&0&0&0  &0     &0     &0       &0         &1           &104         &4020          &106338\\
$t_{m,5}$&0&0&0&0  &0     &0     &0       &0         &0           &0              &20               &1571\\
$t_{m,6}$&0&0&0&0&  0     &0     &0       &0        &0            &0              &0                  &2\\
\hline
$t_{m}$&1&2&6&25&128&758&5014&36194&280433&2303918&19885534&179028087
\end{tabular}
\end{center}
\caption{A table of values for $t_{m,s}$, the number of rectangulations with $m$ tiles and $s$ singular vertices.}\label{fig:data}
\end{figure}

The sequence $t_m$ continues
\begin{align*}
1, &2, 6, 25, 128, 758, 5014, 36194, 280433, 2303918, 19885534, 179028087, 1671644720,\\
 &16114138846, 159761516110, 1623972412726, 16880442523007, 179026930243822,\\
 & 1933537655138482,  21231023519199575, 236674460790503286, 2675162663681345170,\\
 & 30625903703241927542, 354767977792683552908, 4154708768196322925749, \\
 &49152046198035152483150, 587011110939295781585102, 7072674305834582713614923
\end{align*}

This motivates the following surprising conjecture, which we have been unable to prove, though we give partial results in section~\ref{sec:symmetry}.

\begin{conjecture}\label{conj1}
 $t_n\equiv 1\!\!\mod 2$ if $n=8k+1$ or $n=8k+4$. Otherwise $t_n\equiv 0\!\!\mod 2$.  
\end{conjecture}

Getting back to the topology,  the calculations of $t_{m,e}$ above can be used to calculate $k_n$:

\begin{proposition}\label{prop:kn}
The sequence $k_n$ referred to in Theorem~\ref{thm:wedge} is given by:
\begin{align*}
0, &2, 4, 19, 85, 445, 2513, 15221, 97436, 653290, 4554620, 32833261,\\
&243633947, 1854129607, 14428437881, 114522981916, 925229661343, \\
&7594812038558, 63246031323436, 533614085123809, 4556201784167013, \\
&39330233695303765, 342938769382591967, 3018115913779272617, \\
&26790754504125156939, 239715620518047835311, 2160879323839557205915, \\
&19614261422949114679816,
\ldots
\end{align*}
for $n\geq 1$.
\end{proposition}
\begin{proof}
By Proposition~\ref{prop:cell}, $\mathbb T_n$ has a cell decomposition with cells in correspondence with rectangulations. The dimension of  a cell corresponding to a rectangulation with $m$ rectangles and $s$ singular points is $m-s-1$. Hence the Euler characteristic $\chi(\mathbb T_n)= \sum_{m=1}^n\sum_s (-1)^{m-s-1}t_{m,s}$. On the other hand $\chi(\bigvee_{i=1}^{k_n}S^{n-1})=1+(-1)^{n-1}k_n$. Hence $k_{n}=(-1)^{n-1}(-1+\sum_{m=1}^n\sum_s (-1)^{m-s-1}t_{m,s})$.
\end{proof}

{\bf Acknowledgment:} We thank Nathan Reading  for a careful reading of the manuscript, and the referee, who found the inequivalent homeomorphic tilings listed above.

\section{Proof of Theorem~\ref{thm:recursion}}

Every rectangular tiling, except ones with only horizontal edges, can be generated from a simpler tiling by the process in Figure~\ref{fig:cont}, where $c=1$. The simpler tiling is pictured in (A). Then one pushes an edge of length $\ell$ in from the right, blocking $\ell-1$ horizontal edges from hitting the right edge, as in (B). One then adds horizontal edges in the newly created box, some of which create singular vertices as in (C), and some of which do not as in (D). However, some tilings may be generated in more than one way from this process. For example, the tiling 
\begin{minipage}{.4in}
    \begin{tikzpicture} [thick]
    \pgfsetfillopacity{.3}
\draw[fill=brown](0,0) -- (1,0)--(1,1)--(0,1)--(0,0);
\draw (0,.33)--(1,.33);
\draw (0,.66)--(1,.66);
\draw (.5,0)--(.5,.33);
\draw (.5,.66)--(.5,1);
\end{tikzpicture}
\end{minipage}
comes from two different simpler tilings. To take care of this we use an inclusion-exclusion argument and write $$t_{m,r,s}=\sum_{c\geq 1}(-1)^{c+1}(\text{\# of ways to push in $c$ edges from the right from a simpler tiling})$$

First we count the ways to push in $c$ lines from the right with total length $\ell$, as in Figure~\ref{fig:cont} (B). (Note that $\ell=\bar r-r+m-\bar m$, because $\Delta\text{Tiles}=\Delta(\text{Right edges})-\ell$.) Since there are $\bar r+1$ available slots on the right, this is the count of the number of $c$-component subsets of $\bar r+1$ with a total length of $\ell$, which by Lemma~\ref{lem1}, is $\binom{\ell-1}{c-1}\binom{\bar r+2-\ell}{c}$.
 Next, we need to create $(s-\bar s)$ singular vertices, and the only way to do this is to put a horizontal line at one of the existing pushed-in horizontal lines, as in (C). There are $\ell-c$ pushed in lines, so there are $\binom{\ell-c}{s-\bar s}$ choices available. Finally, we need to distribute the remaining horizontal edges to get an $m$-tile configuration which is $r$-heavy. The number of bins these new horizontal lines can go to is $\ell$. Each pushed-in component creates a new tile making $c$, and each singular vertex also creates a new tile, making $c+s-\bar s$. So we need to create $m-\bar m-(c+s-\bar s)$ new tiles. Hence we need to count the number of ways to distribute $m-\bar m-c-s+\bar s$ edges into the $\ell$ distinct slots they can go, as in (D).
By Lemma~\ref{lem2}, this is $\binom{m-\bar m-c-s+\bar s+\ell-1}{\ell-1}$. Thus we have accounted for all four factors of the coefficient in the formula.

The limits of the summations are explained as follows. Given an $\bar r$-heavy tiling, one can push in at most $\lceil(\bar r+1)/2\rceil$ edges. The number of tiles in the simpler tiling must be smaller, so $\bar m$ ranges to $m-1$. The number of edges meeting the right may not be smaller in the simpler tiling, but we can at least say it has to be less than the number of tiles $\bar m$. Finally the number of singular vertices must indeed be less than or equal to the number in the more complex tiling. This completes the proof.

In the following lemma, a \emph{component} of a subset  $X\subset\{1,\ldots,\bar r+1\}$ is a maximal set of numbers of the form $\{i,i+1,i+2,\ldots,i+k\}$ contained in $X$.
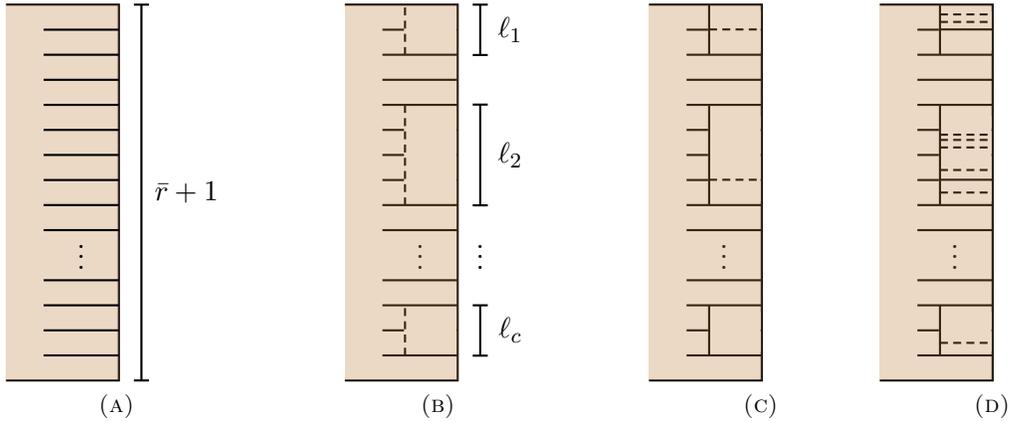
\begin{figure}% 
\centering 
\subfloat[][]{\begin{tikzpicture}[thick]
\draw(0,0)--(1.5,0)--(1.5,5)--(0,5);
\fill[color=brown,opacity=.3](0,0)--(1.5,0)--(1.5,5)--(0,5);
\foreach \h in {1,...,9}
{
\draw (.5,5-\h/3)--(1.5,5-\h/3);
};
\foreach \h in {1,...,4}
{
\draw (.5,\h/3)--(1.5,\h/3);
};
\draw (1,1.75) node {$\vdots$};
\draw (1.7,0)--(1.9,0);
\draw (1.7,5)--(1.9,5);
\draw (1.8,0)--(1.8,5);
\draw (2.4,2.5) node {$\bar r+1$};
\end{tikzpicture}}
\qquad \qquad
\subfloat[][]{\begin{tikzpicture}[thick]
\draw (0,0)--(1.5,0)--(1.5,5)--(0,5);
\foreach \h in {1,...,9}
{
\draw (.5,5-\h/3)--(1.5,5-\h/3);
};
\foreach \h in {1,...,4}
{
\draw (.5,\h/3)--(1.5,\h/3);
};
\draw (1,1.75) node {$\vdots$};
\draw[white,fill=white] (.8,.4) rectangle (1.47,.9);
\draw[ densely dashed] (.8,.33)--(.8,1);
\draw (1.7,.33)--(1.9,.33);
\draw (1.7,1)--(1.9,1);
\draw (1.8,.33)--(1.8,1);
\draw (2.2,.67) node {$\ell_c$};
\draw[white,fill=white] (.8,4.4) rectangle (1.47,4.9);
\draw[ densely dashed] (.8,4.33)--(.8,5);
\draw (1.7,4.33)--(1.9,4.33);
\draw (1.7,5)--(1.9,5);
\draw (1.8,4.33)--(1.8,5);
\draw (2.2,4.67) node {$\ell_1$};
\draw[white,fill=white] (.8,2.4) rectangle (1.47,3.5);
\draw[ densely dashed] (.8,2.33)--(.8,3.67);
\draw (1.7,2.33)--(1.9,2.33);
\draw (1.7,3.67)--(1.9,3.67);
\draw (1.8,2.33)--(1.8,3.67);
\draw (2.2,3.0) node {$\ell_2$};
\draw (1.8,1.75) node {$\vdots$};
\fill[color=brown,opacity=.3](0,0)--(1.5,0)--(1.5,5)--(0,5);
\end{tikzpicture}} 
\qquad\qquad
\subfloat[][]{\begin{tikzpicture}[thick]
\draw (0,0)--(1.5,0)--(1.5,5)--(0,5);
\foreach \h in {1,...,9}
{
\draw (.5,5-\h/3)--(1.5,5-\h/3);
};
\foreach \h in {1,...,4}
{
\draw (.5,\h/3)--(1.5,\h/3);
};
\draw (1,1.75) node {$\vdots$};
\draw[white,fill=white] (.8,.4) rectangle (1.47,.9);
\draw(.8,.33)--(.8,1);
%\draw (1.7,.33)--(1.9,.33);
%\draw (1.7,1)--(1.9,1);
%\draw (1.8,.33)--(1.8,1);
%\draw (2.2,.67) node {$\ell_c$};
\draw[white,fill=white] (.8,4.4) rectangle (1.47,4.9);
\draw (.8,4.33)--(.8,5);
%\draw (1.7,4.33)--(1.9,4.33);
%\draw (1.7,5)--(1.9,5);
%\draw (1.8,4.33)--(1.8,5);
%\draw (2.2,4.67) node {$\ell_1$};
\draw[white,fill=white] (.8,2.4) rectangle (1.47,3.5);
\draw (.8,2.33)--(.8,3.67);
%\draw (1.7,2.33)--(1.9,2.33);
%\draw (1.7,3.67)--(1.9,3.67);
%\draw (1.8,2.33)--(1.8,3.67);
%\draw (2.2,3.0) node {$\ell_2$};
%\draw (1.8,1.75) node {$\vdots$};
\draw[ densely dashed] (.8,4.67)--(1.5,4.67);
\draw[ densely dashed] (.8,2.67)--(1.5,2.67);
\fill[color=brown,opacity=.3](0,0)--(1.5,0)--(1.5,5)--(0,5);
\end{tikzpicture}} 
\qquad\qquad
\subfloat[][]{\begin{tikzpicture}[thick]
\draw (0,0)--(1.5,0)--(1.5,5)--(0,5);
\foreach \h in {1,...,9}
{
\draw (.5,5-\h/3)--(1.5,5-\h/3);
};
\foreach \h in {1,...,4}
{
\draw (.5,\h/3)--(1.5,\h/3);
};
\draw (1,1.75) node {$\vdots$};
\draw[white,fill=white] (.8,.4) rectangle (1.47,.9);
\draw(.8,.33)--(.8,1);
%\draw (1.7,.33)--(1.9,.33);
%\draw (1.7,1)--(1.9,1);
%\draw (1.8,.33)--(1.8,1);
%\draw (2.2,.67) node {$\ell_c$};
\draw[white,fill=white] (.8,4.4) rectangle (1.47,4.9);
\draw (.8,4.33)--(.8,5);
%\draw (1.7,4.33)--(1.9,4.33);
%\draw (1.7,5)--(1.9,5);
%\draw (1.8,4.33)--(1.8,5);
%\draw (2.2,4.67) node {$\ell_1$};
\draw[white,fill=white] (.8,2.4) rectangle (1.47,3.5);
\draw (.8,2.33)--(.8,3.67);
%\draw (1.7,2.33)--(1.9,2.33);
%\draw (1.7,3.67)--(1.9,3.67);
%\draw (1.8,2.33)--(1.8,3.67);
%\draw (2.2,3.0) node {$\ell_2$};
%\draw (1.8,1.75) node {$\vdots$};
\draw (.8,4.67)--(1.5,4.67);
\draw (.8,2.67)--(1.5,2.67);
\draw[ densely dashed] (.8,4.77)--(1.5,4.77);
\draw[ densely dashed] (.8,4.87)--(1.5,4.87);
\draw[ densely dashed] (.8,2.5)--(1.5,2.5);
\draw[ densely dashed] (.8,2.8)--(1.5,2.8);
\draw[ densely dashed] (.8,3.1)--(1.5,3.1);
\draw[ densely dashed] (.8,3.2)--(1.5,3.2);
\draw[ densely dashed] (.8,3.27)--(1.5,3.27);
\draw[ densely dashed] (.8,.5)--(1.5,.5);
\fill[color=brown,opacity=.3](0,0)--(1.5,0)--(1.5,5)--(0,5);
\end{tikzpicture}} 
\caption{(A): The right side of the square with $\bar r$ edges (and therefore $\bar r+1$ tiles) hitting it. (B): Pushing in $c$ vertical edges, of total length $\ell_1+\cdots+\ell_c=\ell$. (C): Adding $s-\bar s$ horizontal line segments to create $s-\bar s$ new  singular vertices. (D): Adding edges to the $\ell$ available bins to bring the number of tiles to $m$.}% 
\label{fig:cont}% 
\end{figure}

\begin{lemma}\label{lem1}
The number of $c$-component subsets of $\{1,\ldots,\bar r+1\}$ of total size $\ell$ is given by the formula
$$\binom{\ell-1}{c-1}\binom{\bar r+2-\ell}{c}.$$
\end{lemma}
\begin{proof}
We think of a subset $X\subset\{1,\ldots,\bar r+1\}$ as a line of $\bar r+1$ billiard balls, some of which are black and indicate membership in $X$, and some of which are white, indicating non-membership in $X$. By hypothesis we are counting the number of such configurations where there are $c$ black components. There are $\binom{\ell-1}{c-1}$ ways to break $\ell$ into $c$ pieces.
The number of ways of distributing those $c$ pieces into the $\bar r+2-\ell$ interstices of the $\bar r +1 -\ell$ white billiard balls is  $\binom{\bar r+2-\ell}{c}$.
\end{proof}
The following lemma is well-known and can be found, for example, in \cite{comb}.
\begin{lemma}\label{lem2}
The number of nonnegative integer solutions $(x_1,\ldots,x_n)$ to the equation $x_1+\cdots+x_n=m$ is given by the formula $\binom{m+n-1}{n-1}$.
\end{lemma}
The recursive strategy of the proof of Theorem~\ref{thm:recursion} is illustrated in Figure~\ref{squarediag}. We start with the tiling with one rectangle, push an edge in from the right in all ways, and then again push an edge in from the right to the resulting tilings, restricting at each step to rectangulations with $\leq 4$ tiles. The new pushed-in edges that are created at each step are dashed. Similarly one can construct new tilings starting with \begin{minipage}{.45in}
    \begin{tikzpicture} [thick]
    \pgfsetfillopacity{.3};
\draw[fill=brown] (0,0) -- (1,0)--(1,1)--(0,1)--(0,0);
\draw (0,.5)--(1,.5);
\end{tikzpicture}
\end{minipage}, as in in Figure~\ref{squarediag2}.

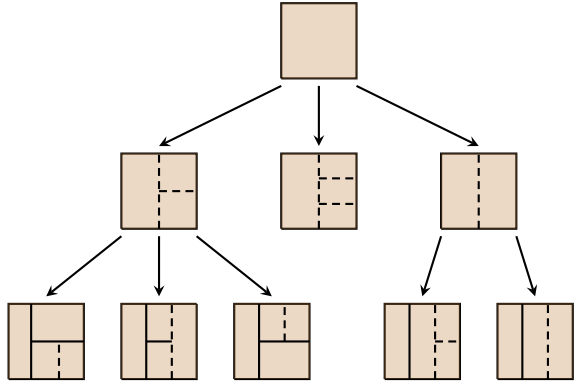
\begin{figure}
\begin{center}
    \begin{tikzpicture} [thick]
\draw(0,0) -- (1,0)--(1,1)--(0,1)--(0,0);
\fill[brown,opacity=.3](0,0) -- (1,0)--(1,1)--(0,1)--(0,0);
\draw(.3,0)--(.3,1);
\draw(.3,.5)--(1,.5);
\draw[densely dashed](.67,0)--(.67,.5);
\draw(1.5,0)-|(2.5,1.0);
\draw(1.5,0)|-(2.5,1.0);
\fill[brown,opacity=.3](1.5,0) -- (2.5,0)--(2.5,1)--(1.5,1)--(1.5,0);
\draw(1.83,0)--(1.83,1);
\draw[densely dashed](2.17,0)--(2.17,1);
\draw(1.83,.5)--(2.17,.5);
\draw(3,0)-|(4,1)-|(3,0);
\fill[brown,opacity=.3](3,0) -- (4,0)--(4,1)--(3,1)--(3,0);
\draw(3.33,0)--(3.33,1);
\draw[densely dashed](3.67,.5)--(3.67,1);
\draw(3.33,.5)--(4,.5);
\draw(1.5,2)-|(2.5,3)-|(1.5,2);
\fill[brown,opacity=.3](1.5,2) -- (1.5,3)--(2.5,3)--(2.5,2)--(1.5,2);
\draw(5,0)-|(6,1)-|(5,0);
\fill[brown,opacity=.3](5,0) -- (6,0)--(6,1)--(5,1)--(5,0);
\draw(6.5,0)-|(7.5,1)-|(6.5,0);
\fill[brown,opacity=.3](6.5,0) -- (7.5,0)--(7.5,1)--(6.5,1)--(6.5,0);
\draw(5.75,2)-|(6.75,3)-|(5.75,2);
\fill[brown,opacity=.3](5.75,2) -- (6.75,2)--(6.75,3)--(5.75,3)--(5.75,2);
\draw(3.625,2)-|(4.625,3)-|(3.625,2);
\fill[brown,opacity=.3](3.625,2) -- (4.625,2)--(4.625,3)--(3.625,3)--(3.625,2);
\draw(3.625,4)-|(4.625,5)-|(3.625,4);
\fill[brown,opacity=.3](3.625,4) -- (4.625,4)--(4.625,5)--(3.625,5)--(3.625,4);
\draw(5.33,0)--(5.33,1);
\draw[densely dashed](5.67,0)--(5.67,1);
\draw[densely dashed](5.67,.5)--(6,.5);
\draw(6.83,0)--(6.83,1);
\draw[densely dashed](7.17,0)--(7.17,1);
\draw[densely dashed](2,2)--(2,3);
\draw[densely dashed](2,2.5)--(2.5,2.5);
\draw[densely dashed](4.125,2)--(4.125,3);
\draw[densely dashed](4.125,2.33)--(4.625,2.33);
\draw[densely dashed](4.125,2.67)--(4.625,2.67);
\draw[densely dashed](6.25,2)--(6.25,3);
\draw[-stealth](3.625,3.9)--(2,3.1);
\draw[-stealth](4.125,3.9)--(4.125,3.1);
\draw[-stealth](4.625,3.9)--(6.25,3.1);
\draw[-stealth](1.5,1.9)--(.5,1.1);
\draw[-stealth](2,1.9)--(2,1.1);
\draw[-stealth](2.5,1.9)--(3.5,1.1);
\draw[-stealth](5.75,1.9)--(5.5,1.1);
\draw[-stealth](6.75,1.9)--(7,1.1);
\end{tikzpicture}
\end{center}
\caption{Illustrating the recursive strategy of the proof of Theorem~\ref{thm:recursion}.}\label{squarediag}
\end{figure}

\begin{figure}
\begin{center}
    \begin{tikzpicture} [thick]
\draw(0,0) -- (1,0)--(1,1)--(0,1)--(0,0);
\fill[color=brown,opacity=.3](0,0) -- (1,0)--(1,1)--(0,1)--(0,0);
\draw(1.5,0)-|(2.5,1.0);
\draw(1.5,0)|-(2.5,1.0);
\fill[color=brown,opacity=.3](1.5,0) -- (2.5,0)--(2.5,1)--(1.5,1)--(1.5,0);
\draw(-1.5,0)-|(-.5,1)-|(-1.5,0);
\fill[color=brown,opacity=.3](-1.5,0) -- (-.5,0)--(-.5,1)--(-1.5,1)--(-1.5,0);
\draw(1.5,2)-|(2.5,3)-|(1.5,2);
\fill[color=brown,opacity=.3](1.5,2) -| (2.5,3)-|(1.5,2);
\draw(6,0)-|(7,1)-|(6,0);
\fill[color=brown,opacity=.3](6,0)-|(7,1)-|(6,0);
\draw(7.5,0)-|(8.5,1)-|(7.5,0);
\draw(9,0)-|(10,1)-|(9,0);
\draw(5.75,2)-|(6.75,3)-|(5.75,2);
\draw(3.625,2)-|(4.625,3)-|(3.625,2);
\draw(3.625,4)-|(4.625,5)-|(3.625,4);
\draw(3.625,0)-|(4.625,1)-|(3.625,0);
\draw(-.5,2)-|(.5,3)-|(-.5,2);
\draw(-2,2)-|(-1,3)-|(-2,2);
\draw(-3.5,2)-|(-2.5,3)-|(-3.5,2);
\draw(7.5,2)-|(8.5,3)-|(7.5,2);
\draw(9,2)-|(10,3)-|(9,2);
\fill[color=brown,opacity=.3](7.5,0)-|(8.5,1)-|(7.5,0);
\fill[color=brown,opacity=.3](9,0)-|(10,1)-|(9,0);
\fill[color=brown,opacity=.3](5.75,2)-|(6.75,3)-|(5.75,2);
\fill[color=brown,opacity=.3](3.625,2)-|(4.625,3)-|(3.625,2);
\fill[color=brown,opacity=.3](3.625,4)-|(4.625,5)-|(3.625,4);
\fill[color=brown,opacity=.3](3.625,0)-|(4.625,1)-|(3.625,0);
\fill[color=brown,opacity=.3](-.5,2)-|(.5,3)-|(-.5,2);
\fill[color=brown,opacity=.3](-2,2)-|(-1,3)-|(-2,2);
\fill[color=brown,opacity=.3](-3.5,2)-|(-2.5,3)-|(-3.5,2);
\fill[color=brown,opacity=.3](7.5,2)-|(8.5,3)-|(7.5,2);
\fill[color=brown,opacity=.3](9,2)-|(10,3)-|(9,2);
\draw[-stealth](3.625,3.9)--(2,3.1);
\draw[-stealth](4.125,3.9)--(4.125,3.1);
\draw[-stealth](4.625,3.9)--(6.25,3.1);
\draw[-stealth](4.125,1.9)--(4.125,1.1);
\draw[-stealth](1.5,1.9)--(-1,1.1);
\draw[-stealth](2,1.9)--(.5,1.1);
\draw[-stealth](2.5,1.9)--(2,1.1);
\draw[-stealth](5.75,1.9)--(6.5,1.1);
\draw[-stealth](6.25,1.9)--(8,1.1);
\draw[-stealth](6.75,1.9)--(9.5,1.1);
\draw[-stealth](3.525,4)--(0,3.1);
\draw[-stealth](3.525,4.1)--(-1.5,3.1);
\draw[-stealth](3.525,4.2)--(-3,3.1);
\draw[-stealth](4.725,4)--(8,3.1);
\draw[-stealth](4.725,4.1)--(9.5,3.1);
\draw(3.625,4.5)--(4.625,4.5);
\draw(-3.5,2.33)--(-3,2.33);
\draw[densely dashed](-3,2)--(-3,3);
\draw[densely dashed](-3,2.67)--(-2.5,2.67);
\draw(-2,2.5)--(-1.5,2.5);
\draw[densely dashed](-1.5,2)--(-1.5,3);
\draw[densely dashed](-1.5,2.5)--(-1,2.5);
\draw(-.5,2.67)--(0,2.67);
\draw[densely dashed](0,2)--(0,3);
\draw[densely dashed](0,2.33)--(.5,2.33);
\draw(1.5,2.5)--(2.5,2.5);
\draw[densely dashed](2.0,2.5)--(2.0,3);
\draw(7.5,2.33)--(8.5,2.33);
\draw[densely dashed](8,2.33)--(8,3);
\draw[densely dashed](8,2.67)--(8.5,2.67);
\draw(9,2.67)--(10,2.67);
\draw[densely dashed](9.5,2)--(9.5,2.67);
\draw[densely dashed](9.5,2.33)--(10,2.33);
\draw(3.625,2.5)--(4.125,2.5);
\draw[densely dashed](4.125,2)--(4.125,3);
\draw(5.75,2.5)--(6.75,2.5);
\draw[densely dashed](6.25,2)--(6.25,2.5);
\draw(-1.5,.5)--(-.5,.5);
\draw(-1.5+.33,.5)--(-1.5+.33,1);
\draw[densely dashed](-1.5+.67,.5)--(-1.5+.67,1);
\draw(0,.5)--(.67,.5);
\draw(.33,.5)--(.33,1);
\draw[densely dashed](.67,0)--(.67,1);
\draw(1.5,.5)--(2.5,.5);
\draw(1.5+.33,.5)--(1.5+.33,1);
\draw[densely dashed](1.5+.67,0)--(1.5+.67,.5);
\draw(3.625,.5)--(3.625+.33,.5);
\draw(3.625+.33,0)--(3.625+.33,1);
\draw[densely dashed] (3.625+.67,0)--(3.625+.67,1);
\draw(6,.5)--(7,.5);
\draw(6.33,0)--(6.33,.5);
\draw[densely dashed](6.67,.5)--(6.67,1);
\draw(7.5,.5)--(8.5,.5);
\draw(7.83,0)--(7.83,.5);
\draw[densely dashed](8.17,0)--(8.17,.5);
\draw(9,.5)--(9.67,.5);
\draw(9.33,0)--(9.33,.5);
\draw[densely dashed](9.67,0)--(9.67,1);
\end{tikzpicture}
\end{center}
\caption{Starting with a different tiling.}\label{squarediag2}
\end{figure}
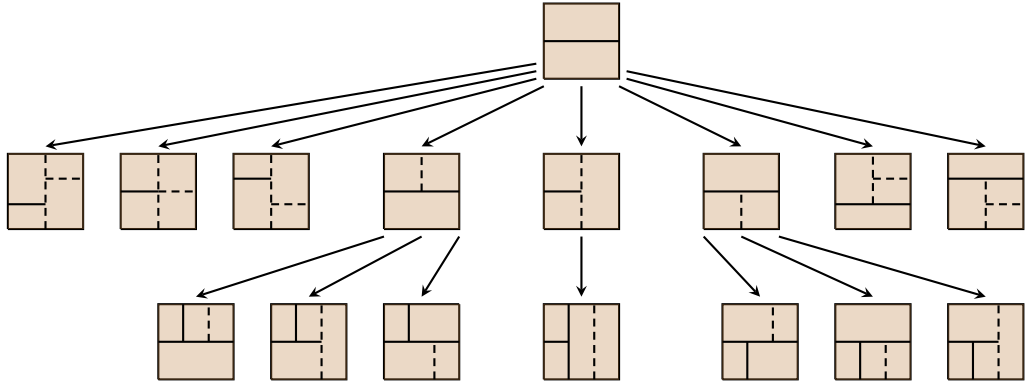

\section{Topological remarks}
 Given a rectangulation $R$ with $\leq n$ tiles, let $e^R\subset\mathbb T_n$ be the set of points in $\mathbb T_n$ which correspond to tilings in the equivalence class $R$. 
As we argue in the next proposition, if $R$ has $m$ tiles and $s$ singular vertices, then $e^R$ is an open cell of dimension $m-s-1$, and the set of all such cells turns
 the space $\mathbb T_n$ into a cell complex.

\begin{proposition}\label{prop:cell}
Let $R$ be a rectangulation with $m$ tiles and $s$ singular vertices. Then $e^R$ is an open cell of dimension $m-s-1$, and the partition of $\mathbb T_n$ into these cells makes it into a cell complex.
\end{proposition}
\begin{proof}
First we define a cell structure on $\mathbb T_n$ as a set, after which we show that the topology induced by this cell structure is equivalent to the one given by the Hausdorff metric.

Given a rectangulation $R$ with $k$ vertices, including vertices on the boundary square and the four corners of the square, fix an ordering of the vertices $v_1,\ldots,v_k$. Let $J$ be a tiling in the equivalence class of $R$. We define $\Phi(J)\in (\mathbb R^2)^k\cong \mathbb R^{2k}$ by the formula $(v_1(J),\cdots, v_k(J))$, where $v_i(J)\in \mathbb R^2$ is the position of the $i$th vertex of $R$ in the particular realization $J$. Then each edge of $J$ imposes a condition of the form 
\begin{equation*}\tag{$\dagger$}
y_i=y_j, \,\,x_i<x_j\,\,\text{ or }\,\, x_i=x_j,\,\, y_i<y_j. \label{eqn}
\end{equation*}
In addition, the four vertices of the ambient square are in fixed positions. The subset $P^R\subset\mathbb R^{2n}$ described by these equalties and inequalties is an intersection of the subspace, $S$, defined by equating some coordinates, with coordinate half-spaces. The closure is a convex subset of $S$ with nonempty interior, and so is a cell of the same dimension as $S$. This dimension is $2k-e-4$, where $e$ is the number of edges. This is because each edge imposes a constraint, and insisting the four corners are in fixed positions reduces the dimension by $4$. We claim that $k=2m+2-s$. To see this, for each rectangle in $R$ draw a diagonal from upper right to lower left. Each $3$-valent vertex is hit by exactly one such diagonal. Two of the four corners of the ambient square are hit by diagonals. Finally every singular vertex is hit twice. Thus counting up the $2m$ diagonal endpoints, we omit two outer vertices and double count each singular vertex, establishing the formula. Since $k-e+m=1$ is the Euler characteristic, we see that 
$$2k-e-4=(2m+2-s)+(k-e)-4=2m+2-s+(1-m)-4=m-s-1.$$
Thus $S$, and therefore $P^R$, has dimension $m-s-1$.

 Now we argue that $\im(\Phi)\subset P^R$. So suppose $J'$ is another tiling in the equivalence class of $R$. Then the homeomorphism sending $J$ to $J'$ preserves horizontal and vertical edges, 
 and we claim it also preserves their directions. That is, the homeomorphism preserves the left-to-right or top-to-bottom ordering of the endpoints of each edge. Because the homeomorphism fixes the corners, it must be orientation-preserving. This  implies that if it fixes the direction of an edge, it will fix the direction of all edges that intersect it. But we know that it preserves the directions of the boundary edges,
 so it fixes the directions of all edges. Thus the vertices and edges of $J'$ lead to the same set of algebraic conditions in (\ref{eqn}), and $\Phi(J')\in P^R$. 
 
Next we show that every $\zeta\in P^R$ corresponds to a tiling of type $R$. Put vertices down in the square at the appropriate coordinates drawn pairwise from $\zeta$, and draw an edge for every one of the pairs of conditions defining the polyhedron. We must show that the resulting object, $J_\zeta$, is a rectangular tiling, in particular, that it is embedded.
 We claim that for every pair of vertices $p,q$, the polyhedral (in)equalities (\ref{eqn}) imply either that one always lies below the other or that one always lies to the left of the other in  $J_\zeta$ for all $\zeta$. To see this, fix a tiling, $J$, that realizes $R$, and
  assume that $p$ lies to the left of $q$ in $J$. 
Draw a NW to SE sawtooth configuration through $p$ and a SW to NE sawtooth through $q$ (Figure ~\ref{sawteeth}). Depending on how these paths intersect or hit the boundary, we get either a monotone left to right path or a monotone down-up path connecting $p$ with $q$, a couple cases of which are illustrated in Figure~\ref{sawteeth}. The inequalities that are associated with the edges of these paths then combine to finish proving the claim. As an immediate result, no two vertices coincide in $J_\zeta$. Suppose now that $J_\zeta$ is not embedded because  a vertex lies in the interior of an edge. Without loss of generality, say the edge is horizontal. Then since the vertex lies at the same height as the endpoints of the edge, the previous analysis implies that its $x$-coordinate always lies between the two edge endpoints. Suppose that in $J$ the vertex lies below the edge. Draw a NW sawtooth from the point, and a SW sawtooth
from the left endpoint of the edge. Because the vertex lies below the edge, these sawtooth paths will combine to produce a monotone increasing path from the vertex to the left end of the edge. Hence the vertex must lie below the edge in $J_\zeta$ based on the resulting string of inequalities. This is a contradiction. A similar argument takes care of the case when two edges intersect orthogonally. Any other type of edge-edge intersection will involve a vertex being contained in an edge's interior. Thus $J_\zeta$, is embedded. By construction, the $1$-skeleton of $J$ is homeomorphic to the $1$-skeleton of $J_\zeta$ by a homeomorphism that preserves directed horizontality and verticality, and we claim this homeomorphism extends to the entire square. Consider a rectangular tile in $J$. Then the boundary is a $1$-dimensional rectangle and maps under the homeomorphism to a $1$-dimensional rectangle in $J_\zeta$. Essentially by the Jordan Curve Theorem, this bounds a disk which only intersects the $1$-skeleton of $J_\zeta$ along its boundary. Thus we can extend the homeomorphism to the interiors of these rectangles. Continue inductively for all rectangles.

So we have a bijection  $\Phi\colon e^R\to P^R$.

In order to get a cell structure, we define a map $\Psi_R\colon\overline{P^R}\to \mathbb T_n$ which is an extension of $\Phi^{-1}$, showing that the image of the boundary $\Psi_R(\partial\overline{P^R})$ lies in the union of lower-dimensional cells. Let $\zeta\in \partial\overline{P^R}$. Then $\zeta$ satisfies the same conditions (\ref{eqn}) as a point in the interior, except that some of the inequalities have become equalities. Geometrically, this corresponds to at least one edge contracting to a point. So $\zeta$ represents a rectangular tiling $J_\zeta$ either with fewer rectangles or more singular points. We define $\Psi_R(\zeta)=J_\zeta$.  Then, $\Psi_R(\partial \overline{P^R})$ is contained in a union of cells of smaller dimension. So the maps $\Psi_R$ induce a cell structure.

Next we argue that the topology induced by this cell structure is equal to the one induced by the Hausdorff metric. We claim that the maps $\Psi_R\colon \overline{P^R}\to \overline{e^R}$ are continuous. The coordinates on $\mathbb R^{2n}$ induce coordinates on $S$ by assembling each set of equated coordinates into a single coordinate. Each coordinate in $S$ thus corresponds to a set of vertices in $e^R$ mutually connected by either vertical or horizontal edges. That is the coordinates in $S$ are in 1-1 correspondence with \emph{walls}. So it suffices to observe that moving a wall (including the case of coalescing two walls or pushing a wall into the boundary square) is continuous in the Hausdorff metric. The continuity of the maps $\Psi_R$ establishes that the cell topology is finer than the metric topology. On the other hand, with respect to the cell complex topology, $\mathbb T_n$ is compact (since it is a union of finitely many cells) and with respect to the metric topology, it is Hausdorff. Any continuous bijection from a compact to a Hausdorff space is a homeomorphism, so this implies that the two topologies indeed coincide.
\end{proof}

\begin{figure}
\begin{center}
\begin{tikzpicture}
\draw[thick] (0,0)-|(3,3)-|(0,0);
\fill[brown,opacity=.3] (0,0)-|(3,3)-|(0,0);
\fill[black] (1,2) circle (.05);
\fill[black] (2.5,1.5) circle (.05);
\draw (.5,0)|-(2,1.2)|-(2.5,1.5)|-(3,1.7);
\draw (0,2.5)-|(1,2)-|(1.4,1)-|(2.2,0);
\draw (.8,2) node {$p$};
\draw (2.7,1.4) node {$q$};
\end{tikzpicture}
\hspace{3em}
\begin{tikzpicture}
\draw[thick] (0,0)-|(3,3)-|(0,0);
\fill[brown,opacity=.3] (0,0)-|(3,3)-|(0,0);
\fill[black] (1,2) circle (.05);
\fill[black] (2.5,1.5) circle (.05);
\draw (.5,0)|-(2,1.2)|-(2.5,1.5)|-(3,1.7);
\draw (0,2.5)-|(1,2)-|(1.5,1.9)--(3,1.9);
\draw (.8,2) node {$p$};
\draw (2.7,1.4) node {$q$};
\end{tikzpicture}
\end{center}
\caption{A case where $p$ is forced to be to the left of $q$ and a case where $p$ is forced to be above $q$ for every $J_\zeta$.}\label{sawteeth}
\end{figure}
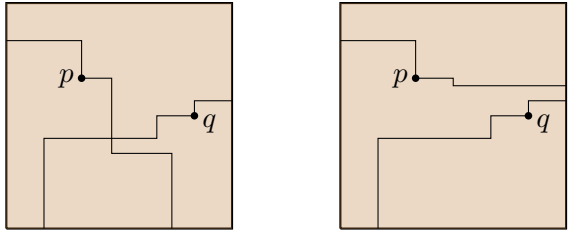

\subsection{Orientation and the cellular boundary operator}
In the proof of Theorem~\ref{thm:wedge} it will be necessary to understand the boundary operator for the cellular chain complex, which we analyze in this subsection. From the proof of the previous proposition, one sees that if $e^{R'}$ appears in the boundary of $e^R$, then it lifts by $\Psi_R$ to possibly several copies of $e^{R'}$ sitting on the boundary of $P^R$. Each of these copies is a face of the polytope $P^R$. Hence in cellular homology, the boundary operator will be a sum over the codimension $1$ faces of $P^R$, with coefficients $\pm 1$ controlled by the orientation. As was mentioned in the proof of Proposition~\ref{prop:cell}, $e^R$ can be parameterized by the positions of its (interior) vertical and horizontal walls. Thus an orientation of $e^R$ regarded as an element of the top exterior power of  the tangent space $\bigwedge^{\dim e^R} T_pP^R$ can be thought of as an ordering of the walls of $R$ up to even permutation. Suppose $\alpha$ is an orientation of $P^R$. We adopt the convention that the induced orientation as you move toward the boundary is given by $\beta$, where $\nu\wedge\beta=\alpha$ and $\nu$ is an outward-pointing normal. 
For example, we consider the rectangulation $R$ given by
$ \begin{minipage}{1.1cm}
    \begin{tikzpicture} [thick,scale=2]
\draw (0,0) -- (.5,0)--(.5,.5)--(0,.5)--(0,0);
\fill[brown,opacity=.3] (0,0)-|(.5,.5)-|(0,0);
\draw (.33,.25)--(.5,.25);
\draw(.33,0)--(.33,.5);
\draw[densely dashed](.17,0)--(.17,.5);
\end{tikzpicture}
\end{minipage}$, and analyze the boundary map when we degenerate the dashed edge by pushing it into the left boundary or into the neighboring right black edge. In both situations, the result is
$R'= \begin{minipage}{1.1cm}
    \begin{tikzpicture} [thick,scale=2]
    \fill[brown,opacity=.3] (0,0)-|(.5,.5)-|(0,0);
\draw (0,0) -- (.5,0)--(.5,.5)--(0,.5)--(0,0);
\draw (.33,.25)--(.5,.25);
\draw(.33,0)--(.33,.5);
\end{tikzpicture}
\end{minipage}$. These are the only two ways in which a copy of $e^{R'}$ appears on the boundary of $P^R$. The cell $e^R$ is parameterized by $x,x',y$ where $x$ and $x'$ are the coordinates of the two vertical walls and $y$ is the coordinate of the horizontal wall.
 Assume that $P^R$ has orientation $ dx \wedge dx'\wedge dy$.  As $x$ moves to the right, the outward pointing normal is $dx$, whereas the outward pointing normal as the dashed line moves to the left is $-dx$. Thus the induced orientations are $ dx'\wedge dy$ and $ -dx'\wedge dy$ respectively. Thus $e^{R'}$ appears with coefficient $0$ in the cellular chain complex. The same considerations apply to the horizontal line, and since these are the only codimension $1$ degeneracies, we conclude that the cell $e^R$ actually represents a cycle in the cellular chain complex: 
$$\partial \,\begin{minipage}{1.1cm}
    \begin{tikzpicture} [thick,scale=2]
\draw (0,0) -- (.5,0)--(.5,.5)--(0,.5)--(0,0);
\fill[brown,opacity=.3] (0,0)-|(.5,.5)-|(0,0);
\draw (.33,.25)--(.5,.25);
\draw(.33,0)--(.33,.5);
\draw(.17,0)--(.17,.5);
\end{tikzpicture}
\end{minipage}=0.$$
(For simplicity we are writing $\partial R$ as opposed to $\partial e^R$.)

In fact, the walls in every rectangulation $R$ can be canonically ordered, giving rise to a canonical orientation of the cell $e^R$. We sketch the construction, leaving details to the reader, since it is not strictly necessary for the main argument. We define two relations $\prec_h$ and $\prec_v$ on the set of walls of $R$ as follows. We let $w\prec_v w'$ if the wall $w$ lies strictly below $w'$ in every tiling realizing $R$. Similarly $w\prec_h w'$ if $w$ lies strictly to the left of $w'$ in every tiling realizing $R$. Suppose that $w$ and $w'$ are distinct vertical walls that are not comparable by $\prec_h$. Then it is an exercise to show that they are comparable by $\prec_v$. So we define a relation $\prec$ on \emph{vertical} walls by
$$w\prec w' \text{ iff }\begin{cases}
w\prec_h w'\text { or } \\
w\not\prec_h w'\text{ and }w'\not\prec_h w\text{ and } w\prec_vw'.
\end{cases}$$
It can be shown that this defines a total order on the set of vertical walls. Define a similar relation on horizontal walls by interchanging $\prec_v$ and $\prec_h$ in the above definition. Finally we get an order on the set of all walls by putting the vertical ones before the horizontal ones.

So now consider \begin{minipage}{1.1cm}
\begin{tikzpicture}[thick,scale=1]
\draw(0,0)--(1,0)--(1,1)--(0,1)--(0,0);
\fill[brown,opacity=.3] (0,0) rectangle (1,1);
\draw(0,.5)--(0,.66);
\draw(.33,0)--(.33,1);
\draw(.66,0)--(.66,1);
\draw(.66,.66)--(1,.66);
\draw(0,.5)--(.66,.5);
\end{tikzpicture}
\end{minipage}.
The boundary cell corresponding to moving the left-hand vertical edge 
\begin{minipage}{1.1cm}
\begin{tikzpicture}[thick,scale=1]
\draw(0,0)--(1,0)--(1,1)--(0,1)--(0,0);
\fill[brown,opacity=.3] (0,0) rectangle (1,1);
\draw(0,.5)--(0,.66);
\draw[densely dashed](.33,0)--(.33,1);
\draw(.66,0)--(.66,1);
\draw(.66,.66)--(1,.66);
\draw(0,.5)--(.66,.5);
\end{tikzpicture}
\end{minipage}
has multiplicity zero in the boundary operator, as in the previous example. There are three other codimension $1$ cells that appear on the boundary, corresponding to moving the right-hand horizontal edge 
\begin{minipage}{1.1cm}
\begin{tikzpicture}[thick,scale=1]
\draw(0,0)--(1,0)--(1,1)--(0,1)--(0,0);
\fill[brown,opacity=.3] (0,0) rectangle (1,1);
\draw(0,.5)--(0,.66);
\draw(.33,0)--(.33,1);
\draw(.66,0)--(.66,1);
\draw[densely dotted](.66,.66)--(1,.66);
\draw(0,.5)--(.66,.5);
\end{tikzpicture}
\end{minipage}
up and down, and the left-hand horizontal edge 
\begin{minipage}{1.1cm}
\begin{tikzpicture}[thick,scale=1]
\draw(0,0)--(1,0)--(1,1)--(0,1)--(0,0);
\fill[brown,opacity=.3] (0,0) rectangle (1,1);
\draw[densely dashed](0,.5)--(.66,.5);
\draw(.33,0)--(.33,1);
\draw(.66,0)--(.66,1);
\draw(.66,.66)--(1,.66);
\end{tikzpicture}
\end{minipage}
down. These three faces have multiplicity $\pm 1$, depending on orientations.  Indeed
$$\partial \,\begin{minipage}{1.1cm}
\begin{tikzpicture}[thick,scale=1]
\draw(0,0)--(1,0)--(1,1)--(0,1)--(0,0);
\fill[brown,opacity=.3] (0,0) rectangle (1,1);
\draw(0,.5)--(0,.66);
\draw(.33,0)--(.33,1);
\draw(.66,0)--(.66,1);
\draw(.66,.66)--(1,.66);
\draw(0,.5)--(.66,.5);
\end{tikzpicture}
\end{minipage}
=
 \begin{minipage}{1.1cm}
\begin{tikzpicture}[thick,scale=1]
\draw(0,0)--(1,0)--(1,1)--(0,1)--(0,0);
\fill[brown,opacity=.3] (0,0) rectangle (1,1);
\draw(0,.5)--(0,.66);
\draw(.33,0)--(.33,1);
\draw(.66,0)--(.66,1);
\draw(.66,.5)--(1,.5);
\draw(0,.5)--(.66,.5);
\end{tikzpicture}
\end{minipage}
-
 \begin{minipage}{1.1cm}
\begin{tikzpicture}[thick,scale=1]
\draw(0,0)--(1,0)--(1,1)--(0,1)--(0,0);
\fill[brown,opacity=.3] (0,0) rectangle (1,1);
\draw(0,.5)--(0,.66);
\draw(.33,0)--(.33,1);
\draw(.66,0)--(.66,1);
\draw(0,.5)--(.66,.5);
\end{tikzpicture}
\end{minipage}
-
 \begin{minipage}{1.1cm}
\begin{tikzpicture}[thick,scale=1]
\draw(0,0)--(1,0)--(1,1)--(0,1)--(0,0);
\fill[brown,opacity=.3] (0,0) rectangle (1,1);
\draw(0,.5)--(0,.66);
\draw(.33,0)--(.33,1);
\draw(.66,0)--(.66,1);
\draw(.66,.5)--(1,.5);
\end{tikzpicture}
\end{minipage}\,.
$$
For example, the sign of the first term can be calculated as follows. Let $x,x',y,y'$ be the coordinates of the $4$ walls in \begin{minipage}{1.1cm}
\begin{tikzpicture}[thick,scale=1]
\draw(0,0)--(1,0)--(1,1)--(0,1)--(0,0);
\fill[brown,opacity=.3] (0,0) rectangle (1,1);
\draw(0,.5)--(0,.66);
\draw(.33,0)--(.33,1);
\draw(.66,0)--(.66,1);
\draw(.66,.66)--(1,.66);
\draw(0,.5)--(.66,.5);
\end{tikzpicture}
\end{minipage}. To get the first term we can let $y$ increase upward to $y'$. The outward normal is then $dy$, so since $dx\wedge dx'\wedge dy\wedge dy'=dy\wedge dx\wedge dx'\wedge dy'$, the induced orientation is $dx\wedge dx'\wedge dy'$ which matches the canonical orientation of this cell.

\subsection{Proof of Theorem~\ref{thm:wedge}}
We define a discrete vector field in the sense of Forman \cite{forman} on the complex $\TT{n}$. To briefly review and establish notation, this is defined to be a collection of pairs of cells $(\alpha,\beta)$, called \emph{vectors}, where $\alpha$ is a codimension $1$ face of $\beta$ such that  $\partial \beta=\pm \alpha+\cdots$ in the cellular homology chain complex.
Every cell of $\TT{n}$ is allowed to appear in at most $1$ pair. Furthermore, we need the vector field to be a gradient field, which is defined to mean that no chain $\alpha_1,\beta_1,\alpha_2,\beta_2,\alpha_3,\beta_3,\ldots$ can loop back on itself, where each $(\alpha_i,\beta_i)$ is a vector from the vector field, and $\alpha_{i+1}$ is a cell in the boundary of $\beta_i$ distinct from $\alpha_i$, with nonzero multiplicity. (That is, $\partial \beta_{i}=r\alpha_{i+1}+\cdots$ where $r\neq 0$.) The critical cells are defined to be those that do not appear in any vector in the vector field. Forman's theorem implies that $\TT{n}$ is homotopy equivalent to a complex which has cells in $1-1$ correspondence with the critical cells.

We will define a gradient vector field which has a single critical $0$-cell, some number of critical $(n-1)$-cells, and no other critical cells. In other words every cell  representing a rectangulation with $m$ tiles, $1<m<n$ appears in some vector $(\alpha,\beta)$ of the vector field. This will establish the theorem, since it implies the complex is homotopy equivalent to a complex with only a $0$-cell and several $(n-1)$-cells.

We define the vector field as follows. Given a rectangulation $R$, define $\rec R$ to be the rectangulation with a new long thin box added on the left of the ambient square.
For example $\rec  \left(\begin{minipage}{.22in}
    \begin{tikzpicture} [thick]
    \fill[brown,opacity=.3] (0,0) rectangle (.5,.5);
\draw (0,0) -- (.5,0)--(.5,.5)--(0,.5)--(0,0);
\draw (0,.25)--(.5,.25);
\end{tikzpicture}
\end{minipage} \right)= \begin{minipage}{.22in}
    \begin{tikzpicture} [thick]
\draw (0,0) -- (.5,0)--(.5,.5)--(0,.5)--(0,0);
 \fill[brown, opacity=.3] (0,0) rectangle (.5,.5);
\draw (.25,.25)--(.5,.25);
\draw(.25,0)--(.25,.5);
\end{tikzpicture}
\end{minipage}$.
 Every nontrivial rectangulation $R$ can be uniquely written $\rec^k S$ for some  $S$ such that $S\neq \rec T$ for any $T$. 
By convention we think of the trivial rectangulation with $1$ tile as $\rec\emptyset$, although $\emptyset$ does not correspond to a rectangulation in $\TT{n}$.
  Create a vector field by forming all possible pairs $(\rec ^{2i} S, \rec^{2i+1}S)$. 
%Let $\partial$ be the boundary operator in the cellular homology chain complex, which is a free $\Z$-module on the set of rectangulations.
With our orientation conventions $\rec^{2i}S$ appears with coefficient $-1$ in $\partial  \rec^{2i+1}S$ because there are $2i+1$ different terms in the boundary that correspond to $\rec ^{2i} S$ which mostly cancel. 
 (On the other hand notice that $\rec^{2i-1}S$ appears with coefficient $0$ in $\partial \rec^{2i}S$.) By design these pairs do not overlap at all. To see there are no closed  loops, notice that in a chain $\alpha_1,\beta_1,\alpha_2,\beta_2$, we must have $\alpha_1=\rec^{2i}S$, $\beta_1=\rec^{2i+1}S$. But then $\alpha_2=\rec^{2i+1}S'$ for some $S'$, because $$\partial\rec^{2i+1}S=- \rec^{2i}S+\sum_{i}r_i\rec^{2i+1}S_i,$$ for some scalars $r_i$.
Since $\alpha_2$ is the first in a pair, $S'=\rec^{2\ell+1}T$ for some $\ell$ and $T$. In particular $\alpha_2$ has more leading rectangles than $\alpha_1$. Since the number of leading rectangles strictly increases along gradient paths, there can be no closed loops.

The critical cells for this vector field consist of the trivial rectangulation, which is the unique $0$-cell, as well as cells corresponding to rectangulations with $n$ tiles that are not of the form $\rec^{2i+1} S$. These are of dimension $n-1$ minus the number of singular vertices. Our next task is to extend the previous vector field to a vector field that includes all singular rectangulations (i.e. rectangulations with at least one singular vertex.) Let $\mathcal C$ be the set of rectangulations corresponding to critical cells from the previous vector field, and define a map $\Delta\colon \mathcal C\to\mathcal C\cup\{0\}$ as follows.  Given a rectangulation $R$, find all positions in the tiling of the form \plussign
or \leftplus.  Consider all the locations which are \emph{furthest right} in the sense that they are geometrically furthest to the right in some tiling realizing $R$. Among all such locations, it is an exercise to show that in every tiling realizing $R$, these locations have the same linear ordering by height, so that we can choose the unique position closest to the top.
 If this position is of the form \leftplus, then we define $\Delta(R)=0$. If it is of the form \plussign then define $\Delta(R)$ to be the rectangulation where this position is changed  to \leftplus. Now the vector field extension is defined to consist of all pairs $(R,\Delta(R))$ where $\Delta(R)\neq 0$, and $e^R$ is critical for the previous vector field. (Which implies that $e^{\Delta(R)}$ is also.) All of these pairs are disjoint since if the upper right instance is \plussign, then it is the first coordinate of a pair, and if it is  \leftplus, then it is the second coordinate of a pair. Similarly, if $\Delta(R)\neq 0$ then $\partial \Delta(R)=\pm R+\cdots$, since $R$ appears only once on the boundary of the cell $\Delta(R)$.
 All singular rectangulations appear either as the first or second coordinate of a pair, so the critical cells are either the unique $0$-cell or are $(n-1)$-dimensional.  Finally we argue there are no closed gradient loops in the combined vector field. We claim that a gradient loop cannot contain any pairs $(\rec^{2i}S,\rec^{2i+1}S)$. If we have $\alpha_2=\rec^{2j}S'$, then $j>i$ and $\beta_2=\rec^{2j+1}S''$. If we have $\alpha_2=\rec^{2j+1}S'$, then this is a contradiction since such rectangulations are always the second coordinate of a vector. So if a gradient path contains a pair $(\rec^{2i}S,\rec^{2i+1}S)$, then all subsequent pairs are of this form, and so by the previous argument, there is no closed loop. So now we can concentrate on pairs $(R,\Delta(R))$ only. 

Notice that $\Delta$ preserves the number of vertical walls, and $\partial$ cannot increase the number. (Terms in the boundary are gotten by collapsing one or more edges to points.)
Thus  a closed gradient loop must have a constant number of vertical walls for every $\alpha_i$ and $\beta_i$. Also, once $\Delta$ operates on a given vertical wall it can never operate on one below it or to the left. So in a loop, it must operate on a single vertical wall. Similarly, the number of edges meeting the wall from the left and right must be constant since $\Delta$ preserves this number and $\partial$ cannot increase it.  

The two types of boundary terms that fix the number of vertical walls and the number of edges meeting those walls from left and right are recorded by the following moves
 $\beta_i\to\alpha_{i+1}$:
$$ \text{\leftplus}\mapsto \text{\plussign}\text{\hspace{3em}}\text{\rightplus}\mapsto\text{\plussign}$$
 However the second move can never be part of a loop, since the left edge starts out below the right edge, and $\Delta$ cannot reverse their order. Hence only the first move is possible in a loop. So consider $\alpha_1,\beta_1,\alpha_2$ where $\alpha_2$ is obtained from $\beta_1$ by the first move above. Then it must have operated on a site below the one that changed from $\alpha_1$ to $\beta_1$, so that \leftplus is its top instance. Thus $\Delta(\alpha_2)=0$ and it cannot be the first cell in a pair. This completes the proof of Theorem~\ref{thm:wedge}.

\section{Symmetric tiles and a mod 2 counting conjecture}\label{sec:symmetry}
Let ${\sf T}_n$ be the set of rectangulations with exactly $n$ tiles. The dihedral group of $8$ elements $D_8$ acts on both ${\sf T}_n$ and $\mathbb T_n$ as the automorphism group of the ambient square. Let $s_n$ be the number of rectangulations fixed by this action. That is, $s_n$ counts the totally symmetric rectangulations.
The next lemma assures us that the fixed points of the action can be realized geometrically.
\begin{lemma}
Any rectangulation which is fixed by the $D_8$ action on ${\sf T}_n$ has a representative by a tiling which is fixed by the $D_8$ action on $\mathbb T_n$.
\end{lemma}
\begin{proof}
By Proposition~\ref{prop:cell}, the set of tilings representing a rectangulation $R$ forms a cell $e^R$. Let $e_\epsilon^R$ be the set of all tilings in $e^R$ where every edge has length $\geq \epsilon$. This is nonempty and contractible  for $\epsilon$ sufficiently small, as it represents a convex subpolyhedron of the polyhedron $P^R$ mentioned in the proof of Proposition~\ref{prop:cell}. Note that $D_8$ acts on $e_\epsilon^R$. 
If $\sigma\in D_8$, $\sigma$ acts on $e^R_\epsilon$, and has a fixed point by Brouwer's fixed point theorem. Taking $\sigma$ to be reflection through a horizontal bisector of the square, $R$ can be realized by a tiling which is symmetric about this horizontal bisector. Thus it is gotten by taking a half-tiling from the top of the square and reflecting it into the bottom of the square. Now, by a similar argument to that which shows $e^R_\epsilon$ is a contractible, the set of all half-tilings of the upper square with edge lengths $\geq \epsilon$ is contractible. This is acted on by the reflection through the vertical line, and this has a fixed point. So $R$ is represented by a tiling which is created by the $\z\times\z\subset D_8$ action applied to a configuration in the upper-right quadrant of the square. Finally, reflection through the SW-NE diagonal acts on the space of all such  tilings of this  quadrant with edges of length $\geq \epsilon$. A fixed point of this reflection is a totally symmetric tiling.
\end{proof}

\begin{lemma}
$s_n\equiv t_n\!\!\mod 2$
\end{lemma}
\begin{proof}
The orbits of the $D_8$ action on ${\sf T}_n$ have an even number of elements except for the singleton orbits.
\end{proof}

\begin{lemma}\label{lem:4k}
A totally symmetric tiling has either $4k$ tiles or $4k+1$ tiles.
\end{lemma}
\begin{proof}
Since the tiling is totally symmetric, $D_8$ acts on the rectangles within the tiling.
The orbit of a tile under the $D_8$ action has either $1$, $4$, or $8$ elements. It has 1 element if and only if the tile contains the square's center in its interior. 
\end{proof}
\begin{proposition}
$s_n=0$ unless $n=4k$ or $n=4k+1$. Furthermore $s_{4k+1}=s_{4k+4}$.
\end{proposition}
\begin{proof}
The first statement follows from Lemma~\ref{lem:4k}.
The bijection corresponding to $s_{4k+1}=s_{4k+4}$ is given by subdividing the central square into $4$ squares. 
\end{proof}
The computed sequence $t_n$ indeed obeys these equations modulo $2$, which acts as a check on our work. Indeed, this sequence appears to satisfy the even stronger property that the number of tilings is even unless $n=8k+1$ or $8k+4$ in which case it is odd. Here is $t_n\!\!\mod 2$, for $1\leq n\leq 28$.
$$1, 0, 0, 1, 0, 0, 0, 0, 1, 0, 0, 1, 0, 0, 0, 0, 1, 0, 0, 1, 0, 0, 0, 0, 1, 0, 0, 1,\ldots$$
Recall Conjecture~\ref{conj1} from the introduction.
\begin{conjecturethree}
 $t_n\equiv 1\!\!\mod 2$ if $n=8k+1$ or $n=8k+4$. Otherwise $t_n\equiv 0\!\!\mod 2$.  
\end{conjecturethree}

Conjecture~\ref{conj1} can be independently verified for small $n$ by directly counting symmetric configurations. Every symmetric tiling is determined by what it looks like in a triangular fundamental domain for the $D_8$ action, depicted in grey in the following picture: \begin{minipage}{.45in}\begin{tikzpicture}[thick]
\draw (0,0)-|(1,1)-|(0,0);
\fill[brown,opacity=.3] (0,0) rectangle (1,1);
\fill[gray](.5,.5)--(1,.5)--(1,1)--(.5,.5);
\end{tikzpicture}\end{minipage}. So we study the possible configurations when restricted to this triangle. It is clear that they must look like
\begin{minipage}{1.1cm}
\begin{tikzpicture}[thick,scale=.66]
\fill[brown,opacity=.3] (0,0)--(1.5,1.5)--(1.5,0);
\draw (0,0)--(1.5,1.5)--(1.5,0);
\draw[densely dashed](0,0)--(.5,0);
\draw[fill=gray] (.5,0)--(.5,.5)--(1,.5)--(1,1)--(1.5,1)--(1.5,0);
\end{tikzpicture}
\end{minipage}
where the grey region is a rectangular tiling, and there are some number of ``sawteeth" that hit the diagonal. The dashed edge may or may not be there, and accounts for the equality $s_{4k+1}=s_{4k+4}$. So for example, here is a count of the symmetric tilings by $17$ rectangles.
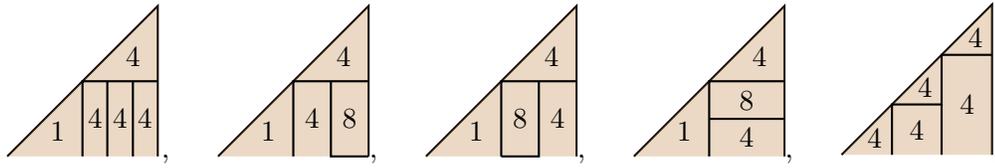
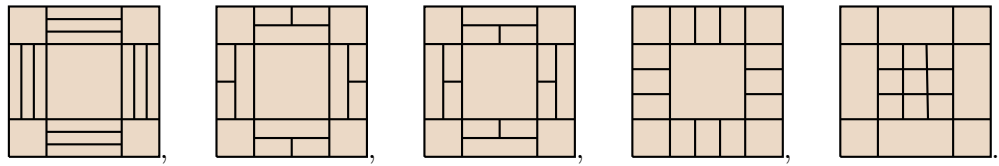
\begin{center}
\begin{tikzpicture}[thick]
\draw(0,0)--(2,2)--(2,0);
\fill[brown,opacity=.3] (0,0)--(2,2)--(2,0);
\draw(1,0)--(1,1)--(2,1);
\draw(1.33,0)--(1.33,1);
\draw(1.67,0)--(1.67,1);
\draw (.67,.33) node {$1$};
\draw (1.17,.5) node {$4$};
\draw (1.5,.5) node {$4$};
\draw (2-.17,.5) node {$4$};
\draw (1.67,1.33) node {$4$};
\end{tikzpicture},\hspace{.5cm}
\begin{tikzpicture}[thick]
\draw(0,0)--(2,2)--(2,0);
\fill[brown,opacity=.3] (0,0)--(2,2)--(2,0);
\draw(1,0)--(1,1)--(2,1);
\draw(2,0)--(1.5,0)--(1.5,1);
\draw (.67,.33) node {$1$};
\draw (1.25,.5) node {$4$};
\draw (1.75,.5) node {$8$};
\draw (1.67,1.33) node {$4$};
\end{tikzpicture},\hspace{.5cm}
\begin{tikzpicture}[thick]
\draw(0,0)--(2,2)--(2,0);
\fill[brown,opacity=.3] (0,0)--(2,2)--(2,0);
\draw(1,0)--(1,1)--(2,1);
\draw(1,0)--(1.5,0)--(1.5,1);
\draw (.67,.33) node {$1$};
\draw (1.25,.5) node {$8$};
\draw (1.75,.5) node {$4$};
\draw (1.67,1.33) node {$4$};
\end{tikzpicture},\hspace{.5cm}
\begin{tikzpicture}[thick]
\draw(0,0)--(2,2)--(2,0);
\fill[brown,opacity=.3] (0,0)--(2,2)--(2,0);
\draw(1,0)--(1,1)--(2,1);
\draw(1,.5)--(2,.5);
\draw (.67,.33) node {$1$};
\draw (1.5,.75) node {$8$};
\draw (1.5,.25) node {$4$};
\draw (1.67,1.33) node {$4$};
\end{tikzpicture},\hspace{.5cm}
\begin{tikzpicture}[thick]
\draw(0,0)--(2,2)--(2,0);
\fill[brown,opacity=.3] (0,0)--(2,2)--(2,0);
\draw(.67,0)--(.67,.67)--(1.33,.67)--(1.33,1.33)--(2,1.33);
\draw(1.33,.67)--(1.33,0);
\draw (.44,.22) node {$4$};
\draw (1.11,.89) node{$4$};
\draw (1.78,1.56) node{$4$};
\draw (1,.33) node {$4$};
\draw (1.67,.67) node {$4$};
\end{tikzpicture},
\end{center}
corresponding to rectangulations
\begin{center}
\begin{tikzpicture}[thick,scale=2]
\pgfsetfillopacity{0.3}
\draw[fill=brown](0,0)--(0,1)--(1,1)--(1,0)--(0,0);
\draw(.25,0)--(.25,1);
\draw(0,.25)--(1,.25);
\draw(.75,0)--(.75,1);
\draw(0,.75)--(1,.75);
\draw(.25,.083)--(.75,.083);
\draw(.25,.166)--(.75,.166);
\draw(.25,.833)--(.75,.833);
\draw(.25,.916)--(.75,.916);
\draw(.083,.25)--(.083,.75);
\draw(.166,.25)--(.166,.75);
\draw(.833,.25)--(.833,.75);
\draw(.916,.25)--(.916,.75);
\end{tikzpicture},\hspace{.5cm}
\begin{tikzpicture}[thick,scale=2]
\pgfsetfillopacity{0.3}
\draw[fill=brown](0,0)--(0,1)--(1,1)--(1,0)--(0,0);
\draw(.25,0)--(.25,1);
\draw(0,.25)--(1,.25);
\draw(.75,0)--(.75,1);
\draw(0,.75)--(1,.75);
\draw(.25,.125)--(.75,.125);
\draw(.5,0)--(.5,.125);
\draw(.25,.875)--(.75,.875);
\draw(.5,.875)--(.5,1);
\draw(.125,.25)--(.125,.75);
\draw(0,.5)--(.125,.5);
\draw(.875,.25)--(.875,.75);
\draw(.875,.5)--(1,.5);
\end{tikzpicture},\hspace{.5cm}
\begin{tikzpicture}[thick,scale=2]
\pgfsetfillopacity{0.3}
\draw[fill=brown](0,0)--(0,1)--(1,1)--(1,0)--(0,0);
\draw(.25,0)--(.25,1);
\draw(0,.25)--(1,.25);
\draw(.75,0)--(.75,1);
\draw(0,.75)--(1,.75);
\draw(.25,.125)--(.75,.125);
\draw(.5,.25)--(.5,.125);
\draw(.25,.875)--(.75,.875);
\draw(.5,.875)--(.5,.75);
\draw(.125,.25)--(.125,.75);
\draw(.25,.5)--(.125,.5);
\draw(.875,.25)--(.875,.75);
\draw(.875,.5)--(.75,.5);
\end{tikzpicture},\hspace{.5cm}
\begin{tikzpicture}[thick,scale=2]
\pgfsetfillopacity{0.3}
\draw[fill=brown](0,0)--(0,1)--(1,1)--(1,0)--(0,0);
\draw(.25,0)--(.25,1);
\draw(0,.25)--(1,.25);
\draw(.75,0)--(.75,1);
\draw(0,.75)--(1,.75);
\draw(.417,0)--(.417,.25);
\draw(.583,0)--(.583,.25);
\draw(.417,.75)--(.417,1);
\draw(.583,.75)--(.583,1);
\draw(0,.417)--(.25,.417);
\draw(0,.583)--(.25,.583);
\draw(.75,.417)--(1,.417);
\draw(.75,.583)--(1,.583);
\end{tikzpicture},\hspace{.5cm}
\begin{tikzpicture}[thick,scale=2]
\pgfsetfillopacity{0.3}
\draw[fill=brown](0,0)--(0,1)--(1,1)--(1,0)--(0,0);
\draw(.25,0)--(.25,1);
\draw(0,.25)--(1,.25);
\draw(.75,0)--(.75,1);
\draw(0,.75)--(1,.75);
\draw(.417,.25)--(.417,.75);
\draw(.583,.25)--(.573,.75);
\draw(.25,.417)--(.75,.417);
\draw(.25,.583)--(.75,.583);
\end{tikzpicture}.
\end{center}

Here the numbers refer to the number of rectangles in the orbit of a given region, and must add up to $17$. Thus we see that  $s_{17}=5$, which is consistent with our calculation that $t_{17}\equiv 1\!\!\mod 2$.

\end{document}